\documentclass[a4paper, 11pt, reqno]{amsart}

\usepackage{amssymb}
\usepackage{amsthm}
\usepackage{bm}
\usepackage{latexsym}
\usepackage{amsmath}
\usepackage{mathrsfs}
\usepackage{caption}
\usepackage{amscd}
\usepackage[all,cmtip]{xy}
\usepackage[usenames]{color}
\usepackage{graphicx}
\usepackage{color}
\usepackage{amscd}
\usepackage{float}
\usepackage{graphics}
\usepackage{tikz}
\usepackage{tikz-cd}
\usepackage{comment}
\usepackage{xspace}
\usepackage{mathtools}
\usepackage{subfig}
\usepackage{enumerate}
\usepackage{hyperref}

\usepackage{marginnote}

\newcommand{\appsection}[1]{\let\oldthesection\thesection
\renewcommand{\thesection}{Appendix \oldthesection}
\section{#1}\let\thesection\oldthesection}

\newtheorem{introcon}{Conjecture}

  \newtheorem{introthm}{Theorem}

  \newtheorem{introcor}{Corollary}

  \newtheorem{theorem}{Theorem}[section]
  \newtheorem{lemma}[theorem]{Lemma}
  \newtheorem{proposition}[theorem]{Proposition}

  \newtheorem{conjecture}[theorem]{Conjecture}

  \newtheorem{definition}[theorem]{Definition}
  \newtheorem{example}[theorem]{Example}

\newtheorem{remark}[theorem]{Remark}

\theoremstyle{remark}

\numberwithin{equation}{section}

\calclayout

\title{Fundamental Groups, Coregularity, \\
and low dimensional klt Calabi Yau pairs}

\author[L.~Braun]{Lukas Braun}
\address{Institut für Mathematik, Fakultät für Mathematik, Informatik und Physik, Universität Innsbruck, Technikerstraße 13, A-6020 Innsbruck, Austria}
\email{lukas.braun@uibk.ac.at}

\thanks{
During the work on this manuscript, LB was partially supported by DFG grant 452847893 and by FWF grant PAT9495823. 
}

\author[F.~Figueroa]{Fernando Figueroa}
\address{Department of Mathematics, Princeton University, Fine Hall, Washington Road, Princeton, NJ 08544-1000, USA
}
\email{fzamora@princeton.edu}


\begin{document}

\begin{abstract} 
In this article, we study how the absolute
coregularity 
of a projective log pair
reflects on
its fundamental group.
More precisely, we conjecture
that for a projective klt log pair $(X,D)$ 
of absolute coregularity $c$ (and arbitrary dimension)
the fundamental group
$\pi_1^{\rm reg}(X,D)$ 
admits a normal abelian subgroup
of finite index
and rank at most $2c$.
We prove this conjecture 
in the cases $0 \leq c \leq 3$, building on the almost abelianity of the fundamental groups of klt Calabi-Yau pairs of dimension $\leq 3$. In the cases $0 \leq c \leq 2$ and fixed dimension, we can furthermore bound the index of a solvable normal subgroup.

In dimension three, we are able to prove almost abelianity of the fundamental group of the regular locus for projective klt Calabi-Yau pairs.

\end{abstract}

\maketitle

\setcounter{tocdepth}{1}
\tableofcontents

\section{Introduction}

In this article, we investigate the connections between two important invariants of projective varieties (or pairs): the \textit{fundamental group} and the \textit{coregularity}.

The \textit{fundamental group} as a topological invariant gains its importance from viewing smooth projective varieties as Kähler manifolds. As such, it is long known that Fano manifolds are simply connected and projective manifolds with trivial (antinef) canonical class have almost abelian (almost nilpotent, respectively) fundamental group. During the last few years, many of these statements have been generalized to orbifolds and to varieties with Kawamata log terminal (klt) singularities, with respect to the \textit{orbifold fundamental group} and/or the \textit{fundamental group of the smooth locus}. While proofs in the orbifold case mostly mimic the differential-geometric proofs from the smooth case using orbifold charts, in the klt case, often new techniques and insights are needed. It is currently not known if klt pairs $(X,D)$ with trivial (antinef) log canonical divisor have almost abelian (almost nilpotent, respectively) fundamental group of the smooth locus $\pi_1(X_{\rm reg})$. 

\subsection{Coregularity and fundamental groups}
The \textit{regularity} has been defined by Shokurov in~\cite{Sho00} and is the highest dimension of a dual complex of a log Calabi-Yau structure. The \textit{coregularity}, as introduced in~\cite{FMP22, Mor22}, measures the difference between the dimension of $X$ and its regularity.
Toric varieties have coregularity zero. More precisely, as~\cite[Thm.~3.25]{Mor22} shows, a variety $X$ with \emph{a torus action of complexity $t$} (i.e. the maximal torus effectively acting on $X$ is of codimension $t$) has coregularity $c \leq t$. So in a sense, the coregularity measures how far a variety is from being toric, compare~\cite[Thm.~3.26]{Mor22}. 

In this context, we pose the following conjecture, where by $\pi_1^{\rm reg}(X,D)$, we mean the orbifold fundamental group of the orbifold pair $(X_{\rm reg},\left. D \right|_{X_{\rm reg}})$, supported on the regular locus $X_{\rm reg}$ of $X$.

\begin{introcon}
\label{introconj-virt-abelian}
    Let $(X,D)$ be a projective klt pair with absolute coregularity $c$, in particular $(X,D)$ must be of Calabi-Yau type. Then $\pi_1^{\rm reg}(X,D)$ is virtually abelian of rank at most $2c$.
\end{introcon}


Though we are not able to prove this conjecture in full generality at the moment, at least the cases $0 \leq c \leq 3$ are proven in the below (in arbitrary dimension) and thus we can provide some strong evidence for the general case.

In particular, the core of the proof is the following reduction to the case of $c$-dimensional \textit{klt Calabi-Yau pairs} (by which we mean nothing else than klt $(X,D)$ with $\mathbb{Q}$-trivial log canonical divisor $K_X+D \sim_{\mathbb{Q}} 0$):

\begin{introthm}\label{introthm-induction-coreg}
    Assume  Conjecture~\ref{introconj-virt-abelian} holds for klt Calabi-Yau pairs of absolute coregularity $c$ and dimension $c$. Then Conjecture~\ref{introconj-virt-abelian} holds for any pair of absolute coregularity $c$.
\end{introthm}

As we already mentioned, it is not known at the moment if smooth loci of Calabi-Yau klt pairs have almost abelian regular fundamental group. However, there is strong evidence not only coming from the manifold (and orbifold) case, but also from the fundamental groups of the whole space in the even dimensional singular case, as~\cite[Thm.~I]{GGK19} shows.

A very recent paper~\cite{GLM23} has investigated (almost \emph{nilpotency} of) the fundamental group in the more general case of \emph{log canonical} Calabi-Yau pairs and obtained definitive results in dimension two. Note that in the klt case and up to dimension two, almost abelianity of $\pi_1^{\rm reg}$ is known due to~\cite{CC14}. We settle the three-dimensional case in the below.

However, in the case of coregularity $0$,  not only Theorem~\ref{introthm-induction-coreg} already implies Conjecture~\ref{introconj-virt-abelian} on its own, we even get a finiteness statement:

\begin{introthm}\label{introthm-coreg0}
Let $(X, D)$ be a projective klt pair with absolute coregularity 0. Then $\pi_1^{\rm reg}(X,D)$ is finite.
\end{introthm}

We subsume the cases of coregularity $1 \leq c \leq 3$ in the following:

\begin{introthm}\label{introthm-coreg12}
Let $1 \leq c \leq 3$ and $(X, D)$ be a projective klt pair with absolute coregularity $c$. Then $\pi_1^{\rm reg}(X,D)$ admits a normal abelian subgroup of finite index and rank at most $2c$.
\end{introthm}

While the rank of the abelian subgroup can be controlled by the absolute coregularity, regardless of the dimension, the index of any abelian subgroup cannot always be controlled even if one fixes both the coregularity and the dimension. 
To deal with this problem, we extend to the class of virtually solvable groups. This yields the following effective bounds:

\begin{introthm}\label{introthm-solv012}
    There are constants $i(c,d)$, depending only on the dimension and coregularity, such that for any klt pair $(X,D)$ of coregularity $0 \leq c \leq 2$, the fundamental group $\pi_1^{\rm reg}(X,D)$ has a solvable subgroup of length at most $2d-1$ and index at most $i(c,d)$.
\end{introthm}

\subsection{Fundamental groups of klt Calabi-Yau pairs and the augmented irregularity}

The \textit{augmented irregularity} $\tilde{q}(X)$ of a normal projective variety $X$ as defined by Kawamata is the maximum of the irregularities $q(X')=h^1(X',\mathcal{O}_X')$ of all finite quasi-\'etale covers $X'\to X$. In the case of manifolds with $K_X=0$, the augmented irregularity is precisely the dimension of the \textit{Albanese variety} - an Abelian variety with fundamental group $\mathbb{Z}^{\tilde{q}(X)}$  - of a respective quasi-\'etale cover, which can be split off and which is the `reason' for the abelian non-finite subgroup in the fundamental group. Conjecturally, this holds in the singular log case as well. In this direction, we state a generalization of~\cite[Corollary.~3.6]{GKP16b} and~\cite[Prop.~7.5, `Torus Cover']{GGK19} (which is stated for projective varieties with canonical singularities and $K_X=0$) to the setting of a log pair:

\begin{introthm}[Torus Covers for pairs]
\label{introthm-torus-covers}
    Let $(X,D)$ be a projective klt pair with $K_X+D \sim_{\mathbb{Q}} 0$. Then there exists a quasi-\'etale cover $Y \to X$, such that $Y= A \times Z$ has a product structure with the following properties:
    \begin{enumerate}
        \item $A$ is an Abelian variety of dimension $\tilde{q}(X)$.
        \item There is a boundary $D_Z$ on $Z$ such that $(Z,D_Z)$ is klt and $K_Z+D_Z\sim_{\mathbb{Q}} 0$. Moreover, $\tilde{q}(Z)=0$.
    \end{enumerate}
\end{introthm}

Obviously, what is missing at this point is to prove finiteness of the fundamental group of the smooth locus in the case that $\tilde{q}(X)=0$. Before we come to this task, we note that combining Theorem~\ref{introthm-torus-covers} with Conjecture~\ref{introconj-virt-abelian}, we get the following connection between coregularity and augmented irregularity.

\begin{introcor}
    Assume Conjecture~\ref{introconj-virt-abelian} holds, then the inequality
    $$
    \tilde{q}(X) \leq {\rm \hat{coreg}}(X,D)
    $$
    holds for  projective klt pairs $(X,D)$ with $K_X+D \sim_{\mathbb{Q}} 0$.
\end{introcor}

While at least in dimensions three and four and $X$ having canonical singularities with $K_X\sim_{\mathbb{Q}} 0$, we know that the fundamental group of the \textit{whole space} is virtually abelian~\cite[Corollary~8.25]{GKP16b}, we did not have any statement yet for the fundamental group of the smooth locus in dimensions greater than two.

Here, combining several of the known methods, we can prove the respective statement in the three-dimensional case\footnote{In the first version of this manuscript, we only dealt with the case of trivial boundary. We sincerely thank Stefano Filipazzi and Mirko Mauri for bringing to our attention that our methods already imply the case of nontrivial boundary.}:

\begin{introthm}
\label{introthm-3dim}
    Let $(X,D)$ be a klt log Calabi-Yau threefold, i.e. $(X,D)$ is klt and $K_X+D \sim_{\mathbb{Q}} 0$.
Let  $\tilde{q}(X)$ be the augmented irregularity of $X$. 
Then, the fundamental group $\pi_1(X_{\rm reg}, D)$ is virtually abelian of rank at most $\max(4, 2\, \tilde{q}(X))$.
\end{introthm}

\subsection*{Acknowledgements}

The authors would like to thank Joaquín Moraga for suggesting the problem and useful discussions,
as well as Stefano Filipazzi, Stefan Kebekus, and Mirko Mauri for very useful comments and references.
FF thanks UCLA for its hospitality during his visit.

\section{Preliminaries}

In this section, we provide the necessary background and prove some preliminary results for the rest of the article.

\subsection{Virtual abelianity and solvability of groups}

\begin{definition}
\label{def:virtually-abelian}
{\em We say that a group $G$ is \textit{virtually abelian of rank $n$} if it has a subgroup of finite index which is abelian of rank $n$. }
\end{definition}

\begin{remark}
{ \em A group $G$ being virtually abelian of rank $n$ is equivalent to having a \textit{normal} subgroup of finite index which is abelian of rank $n$, since one can take the normal core of the abelian subgroup from Definition~\ref{def:virtually-abelian}, which is again abelian of the same rank and has finite index in $G$.}
\end{remark}

\begin{lemma}
\label{lem-virtual-groups}
Let $F, G$ and $H$ be groups fitting in an exact sequence 
$$F \xrightarrow{\phi} G \xrightarrow{\psi} H\rightarrow 1.$$ 
If $H$ is virtually abelian of rank $n$ and $F$ is finite, then $G$ is virtually abelian of rank at most $n$.
\end{lemma}

\begin{proof}
By replacing $F$ with its image $\phi(F)$ we can assume that $F$ is a subgroup of $G$ with $f$ elements.
Since $H$ is virtually abelian, there exists an abelian group $A \leqslant H$ of rank $n$ and of finite index in $H$. Let $A$ be generated by $\{a_i\}_{\{1 \leq i\leq n\}}$. Let $g_i$ be elements in $G$ mapping to $a_i$ via $\psi$. For any $x$ in $F$ and $g_i$, we have that $g_i x g_i^{-1}$ gets mapped to the identity $e\in H$ via $\psi$, hence $g_i x g_i^{-1} \in F$. Thus conjugation by elements  $y \in \langle g_i \rangle$ gives a group action  on $F$. Since $F$ is a finite set, there exists a uniform constant $m$ such that $m$-iterations of all of these actions are the identity, in particular $y^{m} x y^{-m} =x$ for any $x \in F$ and $y \in \langle g_i \rangle$. 

On the other hand, for any $y \in \langle g_i \rangle$, there exists $x\in F$, such that $g_i^{m} y g_i^{-m}=x y$, since it maps to the same element of $A$ as $y$. Therefore 
$$
g_i^{fm}g_j^{fm}g_i^{-fm}=g_i^{(f-1)m}xg_j^{fm}g_i^{-(f-1)m}=xg_i^{(f-1)m}g_j^{fm}g_i^{-(f-1)m}= \ldots=x^{f}g_j^{fm}=g_j^{fm}.
$$
We conclude that $\langle  g_i^{fm}\rangle_i$ is an abelian subgroup of rank $n$ of $G$. This subgroup has finite index in $G$, as it surjects onto $\langle a_i^{fm}\rangle_i$, a finite index subgroup of $H$, via a homomorphism with finite kernel.

More explicitly, let $h_j\langle g_i^{mf}\rangle_i$ be the left cosets in $H$ and for every $j$, $h'_j$ a preimage of $h_j$ in $G$. Then the left cosets of $\langle g_i^{mf}\rangle_i$ in $G$ are of the form $xh'_j\langle  g_i^{mf} \rangle_i$, where $x \in F$, hence there are only finitely many cosets.

\end{proof}

\begin{definition}
    {\em We say that a group $G$ is \textit{virtually solvable of length $n$ and index $i$} if it has a subgroup of index $i$ that is virtually solvable of length $n$.
    When we omit the index $i$, the solvable subgroup only needs to be of finite index.}
\end{definition}

The following two lemmata are counterparts of Lemma~\ref{lem-virtual-groups} for \emph{virtual solvability}, which we need for effective results, cf. Section~\ref{sec:eff-virt-solv}.

\begin{lemma}\label{lem-fin-sol}
    Let $F, G$ and $H$ be groups fitting in an exact sequence 
$$F \xrightarrow{\phi} G \xrightarrow{\psi} H\rightarrow 1.$$ 
If $H$ is virtually solvable of length $n$ and index $i$ and $F$ is finite, then $G$ is virtually solvable of length at most $n+1$ and index at most $i(|F|-1)!$.
\end{lemma}

\begin{proof}
Let $H'$ be the solvable subgroup of $H$ of length $n$ and index $i$. By letting $G'$ be the preimage of $H'$ under $\psi$, we obtain the exact sequence
$$F \xrightarrow{\phi} G' \xrightarrow{\psi} H'\rightarrow 1,$$ 
where $[G:G']=[H:H']=i$.

By taking a quotient if needed, we may assume that $F$ is a subgroup of $G'$.
Let $Z$ be the centralizer of $F$ in $G'$. The index of $Z$ in $G'$ is at most $|{\rm Aut}(F)|$, which is at most $(|F|-1)!$.

Since $F\cap Z$ is abelian and $\psi(Z) \leqslant H'$ is solvable of length $n$, the group $Z$ is solvable of length at most $n+1$. Hence $G$ is virtually solvable of length at most $n+1$ and index at most $i|{\rm Aut}(F)|\leq i(|F|-1)!$.

\end{proof}

\begin{lemma}\label{lem-virt-sol}
        Let $F, G$ and $H$ be groups fitting in an exact sequence 
$$F \xrightarrow{\phi} G \xrightarrow{\psi} H\rightarrow 1.$$ 

If $F$ is virtually solvable of length $n_1$ and index $r_1$, and $H$ is virtually solvable of length $n_2$ and index $r_2$, then $G$ is virtually solvable of length at most $n_1+n_2+1$ and index $r_2(r_1-1)!$
\end{lemma}

\begin{proof}
Let $H'$ be the solvable group of length $n_2$ and index $r_2$ of $H$. By letting $G'$ be the preimage of $H'$ under $\psi$, we obtain the exact sequence:
$$F \xrightarrow{\phi} G' \xrightarrow{\psi} H'\rightarrow 1,$$ 
where $[G:G']=[H:H']=i$.

By taking a quotient if needed, we may assume that $F$ is a subgroup of $G'$.
Let $F'$ be the solvable normal subgroup of $F$ of smallest index. $F'$ is a characteristic subgroup of $F$, otherwise there would be $F''$, another normal subgroup of the same index. But then the group $F'F''$ would be a normal solvable subgroup of smaller index, a contradiction. Therefore $F'$ is a normal subgroup of $G'$.

By Lemma~\ref{lem-fin-sol} applied to the induced exact sequence
$$
F/F' \xrightarrow{\phi} G'/F' \xrightarrow{\psi} H'\rightarrow 1,
$$
the group $G'/F'$ is virtually solvable of length $n_2+1$ and index at most $(r_1-1)!$.
Let the quotient be $\pi:G'\rightarrow G'/F'$. 
Thus there exists a subgroup $H'' \leqslant G'/F'$ solvable  of length $n_2+1$ and index at most $(r_1-1)!$. By taking the preimage $G'':=\pi^{-1}(H'')$, we obtain:

$$F'\rightarrow G'' \xrightarrow{\pi} H'' \rightarrow 1.$$
Therefore $G''$ is solvable of length $n_1+n_2+1$, and its index in $G$ is 
$$[G:G'']=[G:G'][G':G'']\leq r_2(r_1-1)!$$
as stated.    
\end{proof}

\subsection{Generalized pairs and singularities of the Minimal Model Program}

In this subsection we recall the definitions of generalized pairs and the singularities of the Minimal Model Program. For the definition and properties of b-divisors, we refer the reader to \cite[\S 2.3.2]{Corti}.

\begin{definition}
{\em A  \textit{generalized pair} $(X,D,{\bf M})$ is a normal variety $X$, together with an effective $\mathbb{Q}$-divisor $D$ and a nef b-divisor ${\bf M}$, such that $K_X+D+{\bf M}_X$ is $\mathbb{Q}$-Cartier.}
\end{definition}

\begin{definition}
{\em Let $(X, D, {\bf M})$ be a generalized pair and $Y$ a normal variety.
Given a projective birational morphism $f:Y \rightarrow X$, we can define $D_Y$, by the formulas:

$$K_Y+D_Y+{\bf M}_Y=f^*(K_X+D+{\bf M}_X) \text{\ \ \ and \ \ } f_*(D_Y)=D.$$

For a prime divisor $E \subset Y$, we define the \textit{generalized log discrepancy} of $(X,D, {\bf M})$ at $E$ as:

$$a_{E}(X,D, {\bf M})=1-\text{coeff}_E(D_Y).$$

A generalized pair $(X,D, {\bf M})$ is said to be \textit{generalized log canonical} (respectively \textit{generalized kawamata log terminal}), abbreviated glc (respectively gklt) if for every projective birational morphism $Y\rightarrow X$ and every prime divisor $E\subset Y$, the log discrepancy $a_E(X,D,{\bf M})\geq 0$ (respectively $a_E(X,D,{\bf M}) >0$).}
\end{definition}

\begin{definition}
{\em A \textit{generalized log canonical place} (respectively \textit{generalized non-klt place}) of $(X,D,{\bf M})$ is a prime divisor $E$ over $X$ such that $a_E(X,D, {\bf M})=0$ (respectively $a_E(X,D, {\bf M})\leq 0$).

A \textit{generalized log canonical center} (glcc for short) of $(X,D, {\bf M})$ is the image in $X$ of a generalized log canonical place. It is convenient to regard $X$ itself as a glcc, cf.~\cite[Warning~p.~163]{K13} to exclude trivial exceptions to certain statements, see, e.g., Lemma~\ref{lem-image-lcc}.

A generalized pair $(X,D, {\bf M})$ is \textit{generalized divisorially log terminal} (gdlt for short) if there exists an open subset $U \subseteq X$ such that:

\begin{enumerate}
    \item $\text{coeff}(D)\leq 1$,
    \item $U$ is smooth and $D\big|_U$ is simple normal crossing, and
    \item all the generalized non-klt centers of $(X,D, {\bf M})$ intersect $U$ and are given by strata of $\lfloor D \rfloor$.
\end{enumerate}
}
\end{definition}

\begin{definition}
{\em
    Let $(X,D, {\bf M})$ be a generalized log canonical pair and $f:Y\rightarrow X$ a birational morphism, where $K_Y+D_Y+{\bf M}_Y=f^*(K_X+D+{\bf M}_X).$
    We say that $(Y,D_Y,{\bf M})$ is a generalized dlt modification of $(X,D,{\bf M})$ if $(Y,D_Y,{\bf M})$ is generalized dlt and every $f$-exceptional divisor appears in $D_Y$ with coefficient $1$.
}    
\end{definition}

The following lemma is \cite[Theorem 2.9]{FS23}.
\begin{lemma}\label{lem-gdlt-mod}
    Every generalized log canonical pair admits a $\mathbb{Q}$-factorial generalized dlt modification.
\end{lemma}

\begin{definition}
{ \em
We say that a generalized pair $(X,D,{\bf M})$ is \textit{generalized log Fano} if:
\begin{enumerate}
    \item $-(K_X+D+{\bf M}_X)$ is ample, and
    \item $(X,D,{\bf M})$ is generalized klt.
\end{enumerate}

We say that a generalized pair $(X,D,{\bf M})$ is \textit{generalized log Calabi-Yau} if:
\begin{enumerate}
    \item $K_X+D+{\bf M}_X \sim_{\mathbb{Q}} 0$, and
    \item $(X,D,{\bf M})$ is generalized log canonical.
\end{enumerate}

A generalized pair $(X,D,{\bf M})$ is of \textit{generalized Fano type} (respectively \textit{generalized Calabi-Yau type}) if for some choice of $B\geq D$ and some nef b-divisor ${\bf N}$, the generalized pair $(X,B,{\bf M+N})$ is generalized log Fano (respectively generalized log Calabi-Yau).
}
\end{definition}

\begin{definition}
{ \em
    A \textit{contraction} is a projective morphism of quasi-projective varieties $f:X \rightarrow Z$, such that $f_*\mathcal{O}_X=\mathcal{O}_Z$. A \textit{fibration} is a contraction $f:X\rightarrow Z$ such that $\dim X > \dim Z$.
}
\end{definition}

The following is the generalized canonical bundle formula as in \cite{Fil20} and \cite[Theorem 1.2]{HL21}
\begin{proposition}\label{lem-canonical-bundle-formula}
    Let $f:X\rightarrow Z$ be a surjective projective morphism of normal varieties. Let $(X,D, {\bf M})$ be a generalized log canonical pair, such that 
    $$K_X+D+{\bf M}_X \sim_{\mathbb{Q}, f} 0.$$

    Then there exists a generalized log canonical pair $(Z,D_Z, {\bf N})$, such that:

    $$K_X+D+{\bf N}_X \sim_{\mathbb{Q}} f^*(K_Z+D_Z+{\bf N}_Z).$$
\end{proposition}

\begin{remark}
    {\em The pair $(Z,D_Z,{\bf N})$ obtained in Proposition~\ref{lem-canonical-bundle-formula} will be called \textit{the generalized pair induced by the canonical bundle formula}.}
\end{remark}

\begin{lemma}\label{lem-image-lcc}
    Let $f: X \rightarrow Z$ be a projective fibration. Let $(X,D, {\bf M})$ be a generalized projective log canonical pair, and $f^*(K_Z+D_Z+{\bf N}_Z) \sim_{\mathbb{Q}} K_X+D+{\bf M}_X.$ Then the image of a generalized log canonical center of $(X,D, {\bf M})$ is a generalized log canonical center of $(Z, D_Z, {\bf N})$.
\end{lemma}

\begin{proof}
Let $C$ be a glcc of $(X,D, {\bf M})$. If $f(C)=Z$, it is a glcc by convention. 
    Suppose that $f(C) \subset Z$ is not a glcc of $(Z, D_Z, {\bf N})$, i.e. $(Z, D_Z, {\bf N})$ is gklt at $f(C)$.
    Thus we can take $A$, an ample effective divisor in $Z$ containing $f(C)$, general enough such that $(Z,D_Z+A, {\bf N})$ is generalized log canonical. But $(X,D+f^*A, {\bf M})$ is not generalized log canonical at $C$, which contradicts \cite[Theorem 1.6]{Fil20} .
\end{proof}

The following lemma is taken from \cite[Lemma 2.36]{Mor21a}.
\begin{lemma}\label{lem-dlt-mod-Fano-fibration}
    Let $Z \rightarrow W$ be a Fano type fibration. Let $(Z,B, {\bf M})$  be a generalized log Calabi-Yau pair and $(W,B_W, {\bf N})$ be the generalized pair induced by the canonical bundle formula.
    Let $W'\rightarrow W$ be a $\mathbb{Q}$-factorial gdlt modification of $(W,B_W, {\bf N})$. Then there exists a birational contraction $\phi: Z' \dashrightarrow Z$, only extracting generalized log canonical places of $(Z,B, {\bf M})$, such that there exists a Fano type fibration $Z' \rightarrow W'$ making the following diagram commute:

$$\begin{tikzcd}
(Z,B, {\bf M})\ar{d} 
  & (Z',B',{\bf M}) \arrow[l,dashed]{}[swap]{\phi}\ar{d}
\\
(W,B_W, {\bf N})
  & (W',B_{W'}, {\bf N}) \arrow[l]{}
\end{tikzcd}$$
\end{lemma}

The following lemma is a slight improvement of \cite[Lemma 3.7]{MS21}.
\begin{lemma}\label{lem-big-klt}
    Let $(X,D,{\bf M})$ be a generalized klt projective pair and let $A$ be an effective ample $\mathbb{Q}$-divisor on $X$. Assume that $\mathcal{V}$ is a finite set of divisorial valuations for which $a_E(X,D+A, {\bf M}) \in (0,1)$ for all $E\in \mathcal{V}$. Then, there exists an effective $\mathbb{Q}$-divisor $\Delta$ on $X$, such that:
    \begin{enumerate}
        \item $\Delta$ is big,
        \item $(X,\Delta)$ is a klt pair,
        \item $K_{X}+D+A+{\bf M}_X \sim_{\mathbb{Q}} K_{X}+\Delta,$
        \item $\Delta \geq D$, and 
        \item $a_E(X,\Delta)\in (0,1)$ for all $E\in \mathcal{V}$.
    \end{enumerate}
\end{lemma}

\begin{proof}
   Let $f:Y \rightarrow X$ be a log resolution of $(X,D+A,{\bf M})$, such that all the valuations in $\mathcal{V}$ are divisorial in $Y$. Let the log pullback of $(X,D+A,{\bf M})$ be $(Y,D_Y+A_Y,{\bf M}_Y)$. As ${\bf M}_Y+A_Y$ is big and nef, for every $m\in \mathbb{N}$ we can write ${\bf M}_Y+A_Y \sim_{\mathbb{Q}} A_m+\frac{F}{m}$ \cite[Example 2.2.19]{Laz04a}, where $F$ is a fixed effective $\mathbb{Q}$-divisor and $A_m$ is an ample $\mathbb{Q}$-divisor.
   By choosing $k$ large enough and $A_k$ general enough, we obtain a sub-klt pair $(Y,D_Y+A_k+\frac{F}{k})$.
   This pair will satisfy $K_Y+D_Y+A_k+\frac{F}{k}\sim_{\mathbb{Q}} K_Y+D_Y+{\bf M}_Y+A_Y$.
    Every component of $D_Y+A_k+\frac{F}{k}$ with negative coefficients is exceptional over $X$. By hypothesis, for every $E\in \mathcal{V}$, the coefficient $\text{coeff}_E(D_Y+A_Y)\in (0,1)$.
    Thus, setting $A_X=f_*(A_k)$ and $F_X=f_*(\frac{F}{k})$, the pair $(X,D+A_X+F_X)$ is klt with $A_X+F_X\sim {\bf M}+A$.
    Hence $\Delta:=D+A_X+F_X$ has the desired properties.
\end{proof}

The following is a version of \cite[Corollary 1.4.3]{BCHM10} for generalized pairs.
\begin{lemma}\label{lem-BCHM-generalized-pairs}
    Let $(X,D, {\bf M})$ be a generalized projective klt pair, $(X, B, {\bf M})$ be a generalized lc pair, and $\mathcal{V}$ a finite set of generalized log canonical centers of $(X,B,{\bf M})$. 
    Then, there exists a $\mathbb{Q}$-factorial variety $Y$, and a birational morphism $f:Y\rightarrow X$, such that the exceptional divisors of $f$ correspond to the elements of $\mathcal{V}$.
\end{lemma}

\begin{proof}
    Let $A$ be an effective ample divisor general enough. As $(X,D,{\bf M})$ is generalized klt, for $\varepsilon>0$ and $\delta>0$ small enough, we get
    $$
    a_E(X, (1-\varepsilon)B+\varepsilon D+\delta A,{\bf M})\in (0,1)
    $$
    for every $E\in \mathcal{V}$. Thus by Lemma~\ref{lem-big-klt}, there exists a klt pair $(X,\Delta)$ such that $a_E(X, \Delta) \in (0,1)$ for every $E \in \mathcal{V}$. We can conclude by \cite[Corollary 1.4.3]{BCHM10} applied to the klt pair $(X,\Delta)$ and valuations in $\mathcal{V}$.
\end{proof}

\subsection{Coregularity}

In this section we recall the definition of coregularity and absolute coregularity for generalized pairs, and a few lemmata regarding its behavior under the minimal model program.

\begin{definition}
{\em 
Let $(X,B, {\bf M})$ be a generalized log canonical pair. Let $g:(Y,B_Y, {\bf M}) \rightarrow (X,B, {\bf M})$ be a $\mathbb{Q}$-factorial gdlt modification, with $\lfloor B_Y \rfloor=E_1+\ldots +E_r$ an snc divisor.
We define $\mathcal{D}(Y,B_Y,{\bf M})$, the \textit{dual complex} of $(Y,B_Y,{\bf M})$, to be:

\begin{enumerate}
    \item For each $E_i$, there is a vertex $s_{E_i}$ in $\mathcal{D}(Y,B_Y)$,
    \item For any subset $I \subseteq \{1,\ldots, r\}$, and any irreducible component $Z$ of $\cap_{i \in I} E_i$, there is a $(|I|-1)$-dimensional simplex $s_Z$.
    \item For any 
    \begin{itemize}
        \item $I\subseteq \{1, \ldots r\}$,
        \item  $j \in I$,
        \item  any irreducible component $Z$ of $\cap_{i \in I} E_i$,
        \item and $W$ the unique irreducible component of $\cap_{i \in I\setminus \{j\}} E_i$ containing $Z$,
    \end{itemize}
    there is a gluing map given by the inclusion of $s_W$ into $s_Z$ as the unique face that does not contain the vertex $s_{E_j}$.
\end{enumerate}
We define the \textit{dimension of the dual complex} to be the smallest dimension of a maximal simplex in $\mathcal{D}(Y,B_Y,{\bf M})$ (with respect to inclusion).
By \cite[Proposition 2.21]{FMP22}, the dimension of $\mathcal{D}(Y,B_Y, {\bf M})$ does not depend on the gdlt modification.
Hence, we can define the \textit{coregularity} of a generalized pair $(X,B, {\bf M})$, to be:

$$\text{coreg}(X,B, {\bf M}):= \text{dim}\, X-\text{dim}\, \mathcal{D}(Y,B_Y, {\bf M})-1.$$

For a generalized pair $(X,D, {\bf M})$ of log Calabi-Yau type, we define the \textit{absolute coregularity} ${\rm \hat{coreg}}(X,D, {\bf M})$ to be the smallest value of $\text{coreg}(X,B,{\bf M+ N})$, where $(X,B, {\bf M+N})$ is log Calabi-Yau, with $B\geq D$ and $N$ a nef b-divisor. 
}
\end{definition}

The coregularity of a generalized Calabi-Yau pair is preserved after taking a gdlt modification.

\begin{remark}
{\em 


    The absolute coregularity as it is defined here, could be lower than the one in \cite{Mor22}, where the b-divisor does not change for the generalized Calabi-Yau pair defining the coregularity.
    We need to use this definition for the induction in the proof of the main theorems.
}    
\end{remark}

\subsection{Coregularity under the MMP}
The following is \cite[Lemma 2.31]{FMP22}.
\begin{lemma}\label{lem-CY-coreg-MMP}
    Let $(X,B, {\bf M})$ be a generalized Calabi-Yau pair. Let $X \dashrightarrow X'$ be a birational contraction. Denote by $B'$ the push-forward of $B$ on $X'$. Then
    $$
    {\rm coreg}(X,B,{\bf M})={\rm coreg}(X',B', {\bf M}).
    $$
\end{lemma}

The following lemma is \cite[Proposition 3.27]{Mor22} 
\begin{lemma}\label{lem-coreg-bir-morphism}
Let $f:Y \dashrightarrow X$ be a birational contraction, $(X,D, {\bf M})$ a Calabi-Yau type generalized pair, with $f^*(K_X+D+{\bf M}_X)=K_Y+D_Y+{\bf M}_Y$. Assume $D_{Y}\geq 0.$
Then ${\rm \hat{coreg}}(Y,D_Y,{\bf M})\geq {\rm \hat{coreg}}(X,D_X, {\bf M})$.
    
\end{lemma}

\subsection{Orbifold fundamental groups}

In this subsection, we define the orbifold fundamental group of a pair and prove a special case of the main theorems.

\begin{definition}
{\em 
A divisor $D$ is said to have \textit{standard} coefficients, when the coefficients are in the set $\{1-\frac{1}{n}\}_{\{n \in \mathbb{N}\}}\cup \{1\}$.
For an effective $\mathbb{Q}$-divisor $D$, the \textit{standard approximation} $D_{\text{st}}$ is the largest effective divisor with standard coefficients, such that $D_{\text{st}}\leq D$.
}   
\end{definition}

\begin{definition}
{\em
For a generalized  pair $(X,D,{\bf M})$ with $D_\text{st}=\sum_{j \in J} D_j + \sum_{i \in I} (1-\frac{1}{n_i})D_i$, we define the \textit{orbifold fundamental group of the smooth locus} $\pi_1^{\text{reg}}(X,D,{\bf M})$ to be 
$$
\pi_1(X_{\text{reg}}\setminus \text{supp}(D_{\text{st}}))/N,
$$
where $X_{\text{reg}}$ is the smooth locus of $X$ and $N$ is the normal subgroup of $\pi_1(X_{\text{reg}}\setminus \text{supp}(D_{\text{st}}))$ generated by $\gamma_i^{n_i}$, where $\gamma_i$ is a loop around $D_i$ for $i \in I$. 
}
\end{definition}

The following is a version of \cite[Theorem 2]{Bra21} for generalized pairs:

\begin{theorem}\label{thm-Bra21-gen-pair}
    Let $(X,D,{\bf M})$ be a projective generalized klt pair, such that $-(K_X+D+{\bf M}_X)$ is big and nef. Let $\Gamma \leq D$ be an effective divisor with standard coefficients.
    Then, the group $\pi_1^{\text{reg}}(X,\Gamma)$ is finite.
\end{theorem}

\begin{proof}
    As $(X,D, {\bf M})$ is generalized klt weakly Fano, there exists $B$ effective, such that $(X,D+B,{\bf M})$ is generalized klt Calabi-Yau.
    Thus $B$ is big and nef, hence it can be decomposed as $B=A+E$, where $A$ is an ample $\mathbb{Q}$-divisor and $E$ is an effective $\mathbb{Q}$-divisor. 

    Thus for $0<\varepsilon <1$, the divisor $-(K_X+D+\varepsilon A +\varepsilon E +{\bf M}_X)\sim (1-\varepsilon) (A+E)$ is big and nef. By Lemma~\ref{lem-big-klt} applied to the generalized pair $(X,D+\varepsilon E, {\bf M})$ with ample divisor ${\varepsilon A}$, there exists an effective $\mathbb{Q}$-divisor $\Delta$, such that $(X,\Delta)$ is klt and $-(K_X+\Delta) \sim_{\mathbb{Q}} (K_X+D+\varepsilon E+{\bf M}_X+\varepsilon A)$ is big and nef and $\Delta \geq D+\varepsilon E \geq \Gamma$. Therefore by \cite[Theorem 2]{Bra21}, $\pi_1(X,\Gamma)$ is finite.
\end{proof}

\begin{proposition}\label{prop-dim1}
Let $(X,D,{\bf M})$ be a projective generalized klt pair of dimension 1 with standard coefficients, then the following statements hold:

\begin{enumerate}
    \item If $(X,D,{\bf M})$ is of absolute coregularity $0$, then $\pi_1(X,D,{\bf M})$ is finite.
    \item If $(X,D,{\bf M})$ is of absolute coregularity $1$, then $\pi_1(X,D,{\bf M})$ is virtually abelian of rank at most $2$.
\end{enumerate}
\end{proposition}
\begin{proof}
Since $(X,D)$ is klt, then $\lfloor D \rfloor=0$. If $(X,D)$ is of absolute coregularity $0$ and dimension $1$, this means that $(X,D+E)$ is Calabi-Yau for some effective $\mathbb{Q}$-divisor $E$, hence $X$ is isomorphic to $\mathbb{P}^1$. Therefore $(X,D)$ is Fano, and hence by \cite[Theorem 2]{Bra21} the group $\pi_1(X,D_\text{st})$ is finite.

If $(X,D)$ is of absolute coregularity $1$ and dimension $1$, then $(X,D_\text{st})$ is either of absolute coregularity $0$ or $1$, in the first case we can reduce to the previous paragraph. 

If $(X,D_\text{st})$ has absolute coregularity $1$, then either $X$ is an elliptic curve with $D=0$, or $X=\mathbb{P}^1$. In the first case, we have that $\pi_1(X,D_\text{st})=\mathbb{Z}\times \mathbb{Z}$, which satisfies the theorem. 
If $(X,D_\text{st})$ is Fano, then again by \cite[Theorem 2]{Bra21}, we are done. Hence we only need to consider $(\mathbb{P}^1,D)$ klt Calabi-Yau. Therefore the only possibilities for $D$ are $(\frac{1}{2}P_1+\frac{1}{2}P_2+\frac{1}{2}P_3+\frac{1}{2}P_4)$, $(\frac{1}{2}P_1+\frac{2}{3}P_2+\frac{5}{6}P_3)$, $(\frac{1}{2}P_1+\frac{3}{4}P_2+\frac{3}{4}P_3)$ and $(\frac{2}{3}P_1+\frac{2}{3}P_2+\frac{2}{3}P_3)$.
We can give  explicit descriptions of each of the corresponding fundamental groups:

\begin{enumerate}
    \item $\pi_1(\mathbb{P}^1,\frac{1}{2}P_1+\frac{1}{2}P_2+\frac{1}{2}P_3+\frac{1}{2}P_4)\simeq \langle \alpha,\beta,\gamma,\delta \mid \alpha\beta\gamma\delta, \alpha^2,\beta^2,\gamma^2,\delta^2 \rangle$,
    which has a normal abelian subgroup $N=\langle \alpha\beta, \alpha\gamma \rangle \simeq \mathbb{Z}^2$ of index 2.

    \item $\pi_1(\mathbb{P}^1,\frac{1}{2}P_1+\frac{2}{3}P_2+\frac{5}{6}P_3\simeq \langle \alpha,\beta,\gamma \mid \alpha\beta\gamma, \alpha^2,\beta^3,\gamma^6 \rangle$,
    which has a normal abelian subgroup $N=\langle\beta \gamma^4 ,\beta^2\gamma^2\rangle \simeq\mathbb{Z}^2$ of index 6.

    \item $\pi_1(\mathbb{P}^1,\frac{1}{2}P_1+\frac{3}{4}P_2+\frac{3}{4}P_3)\simeq \langle \alpha,\beta,\gamma \mid \alpha\beta\gamma, \alpha^2,\beta^4,\gamma^4 \rangle$,
    which has a normal abelian subgroup $N=\langle\beta \gamma^{3}, \beta ^3 \gamma \rangle \simeq \mathbb{Z}^2$ of index 4.

    \item $\pi_1(\mathbb{P}^1,\frac{2}{3}P_1+\frac{2}{3}P_2+\frac{2}{3}P_3)\simeq \langle \alpha,\beta,\gamma \mid \alpha\beta\gamma, \alpha^3,\beta^3,\gamma^3 \rangle$,
    which has a normal abelian subgroup $N=\langle \alpha\beta^2, \alpha^2\beta  \rangle \simeq\mathbb{Z}^2$ of index 3.
    
\end{enumerate}

\end{proof}

The following statement is \cite[Theorem 2.2]{CC14}.

\begin{proposition} \label{prop-dim2}
    Let $(X,D)$ be a projective klt Calabi-Yau pair of dimension $2$, with standard coefficients, then $\pi_1(X,D)$ is virtually abelian of rank at most $4$.
\end{proposition}

\section{Fundamental Groups under Morphisms}

In this section we prove some results regarding the behavior of the fundamental group after performing certain birational maps.
 
\begin{lemma}\label{lem-dlt-group}
Let $(X,D, {\bf M})$ be a generalized  projective klt pair, such that there exists a generalized Calabi-Yau pair $(X,B,{\bf M})$, with $D \leq B$. 
Let $\phi:Y\rightarrow X$ be a projective morphism, only extracting  generalized log canonical places of $(X,B, {\bf M})$, with $\phi ^*(K_X+B+{\bf M}_X)=K_Y+B_Y+{\bf M}_Y$ and  $\phi ^*(K_X+D+{\bf M}_X)=K_Y+D_Y+{\bf M}_Y$. 

Then, for $\alpha=\frac{1}{m}$, where $m$  is a natural number, large and divisible enough, there exists an effective divisor $P_Y$ on $Y$, such that:
\begin{enumerate}
    \item $D_Y \leq P_Y \leq \alpha D_Y+(1-\alpha)B_Y$, and
    \item there is a surjection of orbifold fundamental groups:
$$\pi_1^{\rm reg}(Y,P_Y, {\bf M})\twoheadrightarrow \pi_1^{\rm reg}(X,D, {\bf M}).$$
\end{enumerate}
   
\end{lemma}

\begin{proof}

We have that $B_Y \geq D_Y$. By hypothesis $B_Y=B'+E$, where $B'$ is the strict transform of $B$ and $E$ is a reduced exceptional divisor. Let $D'$ be the strict transform of $D$.
Thus 
\begin{equation*}
\begin{split}
\pi_1^{\rm reg}(Y,D'+E, {\bf M}) & =\pi_1(Y_{\rm reg}\backslash {\rm supp}\, (D'\cup E) )/N_{D'} \\
& \simeq \pi_1(X_{\rm reg}\backslash { \rm supp}\, (D \cup \phi(E)) )/N_{D} \\
& \twoheadrightarrow \pi_1(X_{\rm reg}\backslash D)/N_D=\pi_1^{\rm reg}(X,D,{\bf M}),
\end{split}
\end{equation*}
where $N_{D'}$ and $N_{D}$ are the normal subgroups generated by elements of the form $\gamma_i^{n_i}$, where $\gamma_i$ are loops around the divisors $D_i$ with coefficient $1-\frac{1}{n_i}$ in the standard approximations of $D'$ and $D$, respectively.

Let us denote the surjection $f:\pi_1^{\rm reg}(Y,D'+E, {\bf M}) \twoheadrightarrow \pi_1^{\rm reg}(X,D,{\bf M})$, induced as displayed above by the surjection $\hat{f}:\pi_1(Y_{\rm reg}\backslash {\rm supp}\, D' \cup E ) \twoheadrightarrow \pi_1 (X_{\rm reg}\backslash {\rm supp}\,  D \cup \phi(E)).$

As $(X,D)$ is a klt pair, for any irreducible component $C$ of $\phi(E)$, we have that the regional fundamental group $\pi^{\rm reg}_1(X,D;C)$ is finite, by \cite[Theorem 1]{Bra21}.  Thus every loop $\gamma$ around $C$ in $\pi^{\rm reg}_1(X,D;C)$, has finite order. We can choose a  constant $q$ not depending on the irreducible component $C$, such that $\gamma^q$ is trivial for every loop.
Therefore $\hat{\gamma}^q$ is trivial in $\pi_1^{\rm reg}(X,D)$, where $\hat{\gamma}$ is the class of the loop $\gamma$ in $\pi_1^{\rm reg}(X,D)$ induced by inclusion.

Therefore the kernel of $\hat{f}:\pi_1(Y_{\rm reg}\backslash {\rm supp}D' \cup E ) \twoheadrightarrow \pi_1 (X_{\rm reg}\backslash {\rm supp} D \cup \phi(E))$ contains all elements $\gamma^{q}$, where $\gamma$ is a loop around an irreducible component of $E$.

Thus the surjection  $f:\pi_1^{\rm reg}(Y,D'+E, {\bf M}) \twoheadrightarrow \pi_1^{\rm reg}(X,D,{\bf M})$ can be extended to:

$$f:\pi_1^{\rm reg}\left(Y,D'+\left(1-\frac{1}{q}\right)E, {\bf M}\right)=\pi_1^{\rm reg}\left(Y,D'+E, {\bf M}\right)/N_{\left(1-\frac{1}{q}\right)E} \twoheadrightarrow \pi_1^{\rm reg}(X,D,{\bf M}).$$

\end{proof}

\begin{lemma}\label{lem-bir-map-group}
Let $(X,D, {\bf M})$ be a projective klt pair, such that there exists a generalized Calabi-Yau pair $(X,B, {\bf M})$, with $D \leq B$.
Let $\phi: Y \dashrightarrow X$ be a birational contraction that only extracts generalized log canonical places from $(X,B, {\bf M})$, with $\phi ^*(K_X+B+{\bf M}_X)=K_Y+B_Y+{\bf M}_Y$ and  $\phi ^*(K_X+D+{\bf M}_X)=K_Y+D_Y+{\bf M}_Y$. 

Then, for $\alpha=\frac{1}{m}$, where $m$  is a natural number, large and divisible enough, there exists an effective divisor $P_Y$ on $Y$, such that:
\begin{enumerate}
    \item $D_Y \leq P_Y \leq \alpha D_Y+(1-\alpha)B_Y$, and
    \item there is a surjection of orbifold fundamental groups:
$$\pi_1^{\rm reg}(Y,P_Y, {\bf M})\twoheadrightarrow \pi_1^{\rm reg}(X,D, {\bf M}).$$
\end{enumerate}

\end{lemma}

\begin{proof}
By Lemma~\ref{lem-BCHM-generalized-pairs}  there exists a birational morphism $f:Y'\rightarrow X$ extracting the same glc places as $\phi: Y \dashrightarrow X$. Let $P_{Y'}$  be defined by the equation
$$g^*(K_Y+P_Y+{\bf M}_Y)=K_{Y'}+P_{Y'}+ {\bf M}_{Y'}.$$
The map $g:Y'\dashrightarrow Y$ is small, hence $\pi_1^{\rm reg}(Y', P_{Y'}, {\bf M})\simeq \pi_1^{\rm reg}(Y,P_Y, {\bf M})$. By replacing $Y$ with $Y'$, we can assume that we have a projective birational morphism $f:Y \rightarrow X$ only extracting generalized log canonical places of $(X,B,{\bf M})$. Thus, we can conclude by applying Lemma~\ref{lem-dlt-group}.
\end{proof}

The following lemma is part of the proof of \cite[Theorem 3.3]{BFMS20} for generalized pairs.
\begin{lemma}\label{lem-MMP-group}
Let $(X,D, {\bf M})$ be a gklt pair. Let $\phi:X \dashrightarrow X'$ be either a flip or a divisorial contraction, with $D':=\phi_*(D')$. Then, there is an induced surjective group homomorphism
$$\pi_1^{\rm reg}(X',D', {\bf M})\twoheadrightarrow\pi_1^{\rm reg}(X,D, {\bf M}).$$
\end{lemma}

\begin{lemma}\label{lem-Fano-extraction}
Let $(F,P_F, {\bf M})$ be a generalized projective klt Fano pair and $(F,B_F, {\bf M})$ be a generalized Calabi-Yau pair.
Let $\phi:F_0 \rightarrow F$ be a morphism only extracting generalized log canonical places of $(F,B_F, {\bf M})$. Define $R_{F_0}=\phi^{-1}_*(P_F)+E,$ where $E_{\geq 0}$ is a $\phi$-exceptional divisor with coefficients less than 1. Then $\pi_1^{\rm reg}(F_0,R_{F_0}, {\bf M})$ is finite.
\end{lemma}

\begin{proof}
    Let $B_{F_0}$ and $P_{F_0}$ be defined by the  equalities:
    $$\phi^{*}(K_F+B_F+{\bf M}_F)=K_{F_0}+B_{F_0}+{\bf M}_{F_0}\text{\ \ \  and \ \ } \phi^{*}(K_F+P_F+{\bf M}_{F})=K_{F_0}+P_{F_0}+{\bf M}_{F_0}.$$

For any $0<\alpha<1$, the divisor $-K_{F_0}-(1-\alpha)B_{F_0}-\alpha P_{F_0}-{\bf M}_{F_0}$ is big and nef. Indeed $$-K_{F_0}-(1-\alpha)B_{F_0}-\alpha P_{F_0}-{\bf M}_{F_0}=-(\alpha(K_{F_0}+B_{F_0}+{\bf M}_{F_0})+(1-\alpha)(K_{F_0}+P_{F_0}+{\bf M}_{F_0})).$$

The  sub-pair $(F_{0},P_{F_0}, {\bf M})$ is generalized sub-klt. Since $\phi$ only extracts glc places of $(F,B_F, {\bf M})$, the pair $(F_0,B_{F_0}, {\bf M})$ is generalized log canonical with $B_{F_0}\geq P_{F_0}$. Hence for $\alpha$ close enough to $0$, the pair $(F_{0},(1-\alpha)B_{F_0}+\alpha P_{F_0}, {\bf M})$ is a gklt weakly Fano pair with $(1-\alpha)B_{F_0}+\alpha P_{F_0}\geq R_{F_0}$. Therefore by Theorem~\ref{thm-Bra21-gen-pair}, $\pi_1^{\rm reg}(F_0,R_{F_0}, {\bf M})$ is finite.

\end{proof}

The following lemma is an enhancement of \cite[Lemma 3.13]{FM23}
\begin{lemma}\label{lem-fibration}
Let $(Z,P)$ be a klt pair, with $\psi:Z\rightarrow W$ a Mori fiber space.
Let $P_\text{st}$ be the standard approximation of $P$. For any prime divisor $E$ on $W$, define $n_{E}$ by $\text{coeff}_{\psi^*(E)}(P_\text{st})=1-\frac{1}{n_E}$, let
$m_{E}$ be the multiplicity of the fiber at $E$. Let $$P_W:=\sum_E \left(1-\frac{1}{m_{E}n_{E}}\right)E.$$

Then we have the following exact sequence:
$$\pi_1^{\rm reg}(F,P\big|_F)\rightarrow \pi_1^{\rm reg}(Z,P)\rightarrow \pi_1^{\rm reg}(W,P_W)\rightarrow 1,$$
where $F$ is a general fiber of $\psi$.
\end{lemma}

\begin{proof}
By \cite[Lemma 1.5.C]{Nor83}, there is a short exact sequence:

$$\pi_1(F_\text{reg}\setminus \text{supp}(P_\text{st}|_{F_\text{reg}})) \rightarrow \pi_1(Z_\text{reg} \setminus \text{supp}(P_\text{st})) \rightarrow \pi_1 (W_\text{reg} \setminus \text{supp}(P_W)) \rightarrow 1.$$

This induces a commutative diagram: \[
\xymatrix{
\pi_1(F_\text{reg}\setminus \text{supp}(P_\text{st}|_{F_\text{reg}})) \ar[r]^-{\psi_1}
\ar[d]^-{\phi_F}
&
\pi_1(Z_\text{reg} \setminus \text{supp}(P_\text{st}))\ar[r]^-{\psi_2}
\ar[d]^-{\phi_X} 
&
\rightarrow \pi_1 (W_\text{reg} \setminus \text{supp}(P_W))
\ar[d]^-{\phi_C}\ar[r]
& 
1\\
\pi_1^{\rm reg}(F,P|_{F})\ar[r]^-{\phi_1} 
&
\pi_1^{\rm reg}(Z,P)\ar[r]^-{\phi_2}
&
\pi_1^{\rm reg}(W,P_W)\ar[r] 
&
1
}
\]

The vertical maps are surjective by the definition of the orbifold fundamental groups of pairs. 

Let $\alpha \in \pi_1^{\rm reg}(Z,P)$ be an element such that $\phi_2(\alpha)=1$. We take $\alpha'$ a preimage of $\alpha$ by $\phi_X$. Then  $\psi_2(\alpha')\in N'$, i.e. $\psi_2(\alpha')=\Pi_E \, g_E \gamma_E^{m_En_E} g_E^{-1}$, where $\gamma_E$ is a loop around $E$.

Let $\gamma_{\psi^{-1}(E)}$ be a loop around $\psi^{-1}(E)$, then there exists $h_E' \in \pi_1(Z^{\rm reg}\setminus P_{\rm st})$, such that $\psi_2(h_E'\gamma_{\psi^{-1}(E)}h_E'^{-1})=g_E \gamma_E^{m_E} g_E^{-1}$. Indeed $\psi_2(\gamma_{\psi^{-1}(E)})=g \gamma_{E}^{m_E} g^{-1}$ for some $g\in \pi_1(W_{\text{reg}}\backslash \text{supp}(P_W))$, since loops around a divisor are well-defined up to conjugation. As $\psi_2$ is surjective, there exists $h'_{E}$ mapping to $g_E g^{-1}$ in $\pi_1(W_{\text{reg}}\backslash \text{supp}(P_W)).$ 

Therefore taking $ \gamma':= \pi_E h_E'\gamma_{\psi^{-1}(E)}^{n_E}h_E'^{-1}$, we have that $\psi_2(\gamma')=\psi_2(\alpha')$. Hence, by exactness of the top row, there exists $\beta \in \psi_1(F\setminus P_{\rm st}|_F)$, such that $\psi_1(\beta')=\alpha'\gamma'^{-1}$.
By definition, $\gamma'$ is in ${\rm Ker}(\psi_X)$, thus we have: $$\psi_X(\psi_1(\beta'))=\psi_X(\alpha'\gamma'^{-1})=\alpha.$$
Therefore $\alpha=\phi_1(\psi_F(\beta'))$, and thus ${\rm Im}(\phi_1) \supseteq {\rm Ker}(\phi_2)$.

Assume $\alpha \in \pi_1^{\rm reg}(Z,P)$ is of the form $\alpha=\phi_1(\beta)$. Since the vertical morphisms are surjective, there exists $\beta'\in \pi_1(F\setminus P_{\rm st})|_F$ such that $\psi_F(\beta')=\beta$. Hence $$\phi_2(\alpha)=\psi_C\circ \psi_2\circ \psi_1(\beta')=1.$$ Thus ${\rm Im}(\phi_1) \subseteq {\rm Ker}(\phi_2)$. Similarly, we get that $\psi_2$ is surjective. Hence the bottom sequence is also exact.
\end{proof}

\section{Fundamental Groups and Coregularity}

In this section we prove the main theorems stated in the introduction, i.e. we prove the cases $c \in \{0,1,2\}$ of Conjecture~\ref{introconj-virt-abelian}, which we restate for the convenience of the reader:

\begin{conjecture}[Conjecture~\ref{introconj-virt-abelian}]
\label{conj-virt-abelian}
    Let $(X,D)$ be a projective klt pair with absolute coregularity $c$. Then $\pi_1(X,D)$ is virtually abelian of rank at most $2c$.
\end{conjecture}

We also restate a slight enhancement of Theorem~\ref{introthm-induction-coreg} - note that we only require pairs with \emph{standard coefficients}:

\begin{theorem}
\label{thm-induction-coreg}
    Assume  Conjecture~\ref{conj-virt-abelian} holds for klt Calabi-Yau pairs with standard coefficients of absolute coregularity $c$ and dimension $c$. Then Conjecture~\ref{conj-virt-abelian} holds for any pair of absolute coregularity $c$.
\end{theorem}

\begin{proof}
We proceed by induction on the dimension, assume that $(X,D)$ is a projective klt pair of absolute coregularity $c$, with $\dim X=d$, and that Conjecture~\ref{conj-virt-abelian} holds for $\dim X < d$. 

As $\pi_1^{\rm reg}(X,D)=\pi_1^{\rm reg}(X,D_\text{st})$ and $\hat{\text{coreg}}(X,D_\text{st})\leq \hat{\text{coreg}}(X,D)$ it is enough to prove the statement for pairs with standard coefficients. Thus, we may assume that $D=D_{st}$
\\

\noindent
\textit{Case 1:} In this case, we prove the statement when $(X,D)$ is a Calabi-Yau pair. \\

If $(X,D)$ is a Calabi-Yau pair, then the dimension of $X$ is $c$ as $(X,D)$ is a klt pair. Since $D$ has standard coefficients, the group $\pi_1^{\rm reg}(X,D)$ is finite by the statement of the theorem.\\

\noindent
\textit{Case 2:} In this case, we prove the statement when $(X,D)$ is not Calabi-Yau. \\

If $(X,D)$ is not Calabi-Yau, then there exists a generalized Calabi-Yau pair $(X,B, {\bf M})$ of coregularity $c$ with  $B \geq D$  and $B-D+{\bf M}_X \not\equiv 0$.

Let $\phi: (Y,B_Y,{\bf M}) \rightarrow (X,B,{\bf M})$ be a $\mathbb{Q}$-factorial gdlt modification. Let $\phi^*(K_X+D+{\bf M}_X)=K_Y+D_Y+{\bf M}_Y$.
Then by Lemma~\ref{lem-dlt-group}, we can choose $P_Y$, such that $\pi_1^{\rm reg}(Y,P_Y, {\bf M} ) \twoheadrightarrow \pi_1^{\rm reg}(X,D, {\bf M})$.

We have that $(Y,B_Y, {\bf M})$ is generalized Calabi-Yau of coregularity $c$ and that $(Y,D_Y,{\bf M})$ is generalized sub-klt.
As by Lemma~\ref{lem-dlt-group}, $D_Y\leq P_Y \leq \alpha D_Y+(1-\alpha)(B_Y)$ for appropriate $\alpha$. For $\alpha$ small enough, $(Y,P_Y,{\bf M})$ is generalized klt of Calabi-Yau type with absolute coregularity at most $c$. In fact it has absolute coregularity exactly $c$, as $P_Y\geq D_Y$ and Lemma~\ref{lem-coreg-bir-morphism} imply that $${\rm \hat{coreg}}(Y,P_Y, {\bf M})\geq {\rm \hat{coreg}}(Y,D_Y, {\bf M})\geq {\rm \hat{coreg}}(X,D, {\bf M})=c.$$

Since $K_Y+P_Y \sim -\alpha B_Y+\alpha D_Y -{\bf M}_Y$ is not pseudo-effective, we can run a $(K_Y+P_Y)$-MMP with scaling of an ample divisor  that ends on a Mori fiber space $f:Z \rightarrow W$.  If $W$ is a point, then $Z$ is of Fano type of Picard number $1$ and has a finite regional fundamental group by~\cite{Bra21}.

So assume $\mathrm{dim}\, W > 0$. By monotonicity of log discrepancies and Lemma~\ref{lem-CY-coreg-MMP}, the generalized pair $(Z,P_Z, {\bf M})$ is again glc of absolute coregularity at most $c$.

By Lemma~\ref{lem-MMP-group}, we have  $\pi_1^{\rm reg}(Z,P_Z, {\bf M})\twoheadrightarrow \pi_1^{\rm reg}(Y,P_Y, {\bf M})$, where $\varphi:Y \dashrightarrow Z$ induces $\varphi_*(P_Y)=P_Z$ and $\varphi_*(B_Y)=B_Z$.

Let $(W,B_W, {\bf N})$ be the generalized pair induced by the canonical bundle formula \ref{lem-canonical-bundle-formula}, i.e. $$f^*(K_W+B_W+{\bf N}_W) \sim_{\mathbb{Q}} K_Z+B_Z +{\bf M}_Z.$$

Let $(W',B_{W'}, {\bf N})$ be a $\mathbb{Q}$-factorial gdlt modification of $(W,B_W, {\bf N})$. By Lemma~\ref{lem-dlt-mod-Fano-fibration}, we have the following commutative diagram:

$$\begin{tikzcd}
(Z,B_Z, {\bf M})\ar{d}{f} 
  & (Z',B_{Z'}, {\bf M}) \arrow[l,dashed]{}[swap]{g}\ar{d}{f'}
\\
(W,B_W, {\bf N})
  & (W',B_{W'}, {\bf N}), \arrow[l]{}
\end{tikzcd}$$
where $g$ is a birational contraction only extracting generalized log canonical places and $f'$ is of Fano type.

By construction, $(Z',B_{Z'}, {\bf M})$ has a generalized log canonical center of dimension at most $c$. By Lemma~\ref{lem-image-lcc}, its image is a glcc of $(W',B_{W'},{\bf N})$ of dimension at most $c$. Since $(W',B_{W'},{\bf N})$ is gdlt, we have $c\geq \hat{\rm coreg}(W', B_{W'},{\bf N})\geq \hat{\rm coreg}(W', B_{W'})$.

By Lemma~\ref{lem-bir-map-group}, there exists an effective divisor $R_{Z'}$, satisfying that $P_{Z'} \leq R_{Z'} \leq B_{Z'}$, where $K_{Z'}+P_{Z'}+{\bf M}_{Z'}=g^*(K_Z+P_{Z}+{\bf M}_Z)$, such that:

$$\pi_1^{\rm reg}(Z',R_{Z'},{\bf M})\twoheadrightarrow \pi_1^{\rm reg}(Z,P_Z, {\bf M}).$$

Let $R_{W'}$ be as in Lemma~\ref{lem-fibration}, applied to $(Z',R_{Z'},{\bf M})$, and $F'$ ($F$, respectively) the general fiber of $f'$ (of $f$, respectively). Consequently, we obtain the exact sequence 
$$
\pi_1^{\rm reg}(F',R_{F'}, {\bf S}) \rightarrow \pi_1^{\rm reg}(Z',R_{Z'},{\bf M}) \rightarrow \pi_1^{\rm reg}(W',R_{W'}, {\bf N}) \rightarrow 1,
$$
where we define $R_{F'}:=R_{Z'}\big |_{F'}$ and ${\bf S}={\bf M}\big |_{F'}$.
Similarly we define  $B_{F}:=B_Z \big |_{F}$ .

The divisor $R_{W'}$ satisfies $R_{W'}\leq B_{W'}$ and ${\rm coeff}(R_{W'})<1$. Since $(W',B_{W'}, {\bf N})$ is gdlt of coregularity at most $c$, the pair $(W',R_{W'}, {\bf N})$ is gklt of absolute coregularity at most $c$.
By the induction hypothesis, $\pi_1^{\rm reg}(W',R_{W'}, {\bf N})=\pi_1^{\rm reg}(W', R_{W'})$ is virtually abelian of rank at most $2c$.

By Lemma~\ref{lem-BCHM-generalized-pairs} applied to the generalized pair $(F,B_{F},{\bf S})$ and the generalized log canonical places extracted by $g\big|_{F'}:F'\dashrightarrow F$, there exists a birational morphism $p:F_{0}\rightarrow F$ extracting the same generalized lc places. Therefore the induced birational map $q:F_{0} \dashrightarrow F'$ is small. Hence 
$$
\pi_1^{\rm reg}(F',R_{F'}, {\bf S})\simeq \pi_1^{\rm reg}(F_0,R_{F_0},{\bf M}\big|_{F_0}),
$$ 
where $R_{F_0}=q^{-1}_*(R_{F'})=p^{-1}_*(R_{F})=p^{-1}_*(P_{F})+E$ satisfies the hypotheses of Lemma~\ref{lem-Fano-extraction}. Therefore $\pi_1^{\rm reg}(F',R_{F'}, {\bf S})$ is finite.

Hence, by Lemma~\ref{lem-virtual-groups}, the group $\pi_1^{\rm reg}(Z',R_{Z'}, {\bf M})$ is also virtually abelian of rank at most $2c$. 
By the surjections 
$$
\pi_1^{\rm reg}(Z',R_{Z'}, {\bf M})\twoheadrightarrow \pi_1^{\rm reg}(Z,P_Z, {\bf M}) \twoheadrightarrow\pi_1^{\rm reg}(Y,P_Y, {\bf M}) \twoheadrightarrow \pi_1^{\rm reg}(X,D, {\bf M})\simeq \pi_1^{\rm reg}(X,D),
$$ 
a virtually abelian group of rank at most $2c$ surjects onto $\pi^{\rm reg}(X,D)$. Therefore $\pi^{\rm reg}(X,D)$ is also virtually abelian of rank at most $2c$.

\end{proof}

\begin{proof}[Proof of Theorem \ref{introthm-coreg0}]
Conjecture~\ref{conj-virt-abelian} is vacuous in dimension 0 and coregularity 0. Therefore Theorem~\ref{thm-induction-coreg} implies that $\pi_1^{\rm reg}(X,D)$ is finite.
\end{proof}

\begin{proof}[Proof of Theorem \ref{introthm-coreg12}]
Proposition~\ref{prop-dim1} tells us that  Conjecture~\ref{conj-virt-abelian} holds for $1$-dimensional klt Calabi-Yau pairs of coregularity 1  with standard coefficients, while Proposition~\ref{prop-dim2} and Theorem~\ref{introthm-3dim} tell us the same for dimension and coregularity $2$ and $3$, respectively.
Therefore, Theorem~\ref{thm-induction-coreg} implies that $\pi_1^{\rm reg}(X,D)$ is virtually abelian of rank at most $2c$ for $1 \leq c \leq 3$.
\end{proof}

\section{Effective Bounds}
\label{sec:eff-virt-solv}

Theorem~\ref{introthm-coreg0} cannot be made effective, even in the case of dimension $1.$

\begin{example}{\em 
    The pair $(\mathbb{P}^1,\frac{n-1}{n}p_1+\frac{n-1}{n}p_2)$ has absolute coregularity $0$ and has fundamental group:
    $\pi_1^{\rm reg}(\mathbb{P}^1,\frac{n-1}{n}p_1+\frac{n-1}{n}p_2)\simeq \mathbb{Z}/n\mathbb{Z},$ which can be arbitrarily large. Note that though this group is abelian, it is \textit{not} the abelian subgroup `obtained' by Theorem~\ref{introthm-coreg0}, as in the case of coregularity $0$ the latter is trivial.}
\end{example}

The case of absolute coregularity $1$ and dimension $1$ yields only the groups: $\mathbb{Z}\times \mathbb{Z}$, $\mathbb{Z}/m\mathbb{Z}$, $A_4$, $A_5$ or $S_4$. Any of these groups admits an abelian subgroup of rank at most $2$ of index at most $60$. Hence there is an effective bound for Theorem~\ref{introthm-coreg12} in dimension $1$.

In general one cannot obtain an effective bound for the finite index in Theorem~\ref{introthm-coreg12} on the index of the abelian subgroups, even if we fix the dimension and the coregularity.

In \cite[Example 7.3]{GLM23}, an example of a family of klt Calabi-Yau surface with fundamental groups, whose abelian subgroups can have arbitrarily large index. Thus in coregularity $2$ an effective bound cannot be made even in dimension $2$.

But if we relax the characterization of the subgroup to be solvable, we conjecture the following:

\begin{conjecture}\label{conj-solv-bound}
    There exists a constant $i(d,c)$ depending on the dimension and coregularity, such that any $klt$ Calabi-Yau type pair $(X,D)$ of dimension $d$ and absolute coregularity $c$ has fundamental group $\pi_1^{\rm reg}(X,D)$ virtually solvable of length at most $2d-1$ and index at most $i(d,c)$.
\end{conjecture}

We can prove Conjecture~\ref{conj-solv-bound} in low coregularity due to the following proposition.

\begin{proposition}\label{prop-induction-solvable}
    Assume Conjecture~\ref{conj-solv-bound} holds for dimension at most $d$, then Conjecture~\ref{conj-solv-bound} holds for absolute coregularity at most $d$ and any dimension.
\end{proposition}

\begin{proof}
    We proceed by induction on the dimension. Assume that $(X,D)$ is a projective klt pair of absolute coregularity $c$, with $\dim X=d$, and that Conjecture~\ref{conj-solv-bound} holds for absolute coregularity $c$ and $\dim < d$.

    By the proof of Theorem \ref{thm-induction-coreg}, we have a surjection
    $$\pi_1^{\rm reg}(Z',R_{Z'}, {\bf M})\twoheadrightarrow \pi_1^{\rm reg}(X,D),$$
    and an exact sequence
$$
\pi_1^{\rm reg}(F',R_{F'}, {\bf S}) \rightarrow \pi_1^{\rm reg}(Z',R_{Z'},{\bf M}) \rightarrow \pi_1^{\rm reg}(W',R_{W'}, {\bf N}) \rightarrow 1,
$$
where $(F',R_{F'}, {\bf S})$ is Fano of dimension at most $d$. Hence by \cite[Theorem 3]{BFMS20} there exists a constant $c(d)$ depending only on $d$, such that $\pi_1^{\rm reg}(F',R_{F'}, {\bf S})$ has an abelian subgroup of index $c(d)$.
$(W',R_{W'}, {\bf N})$ is klt of Calabi-Yau type of coregularity at most $c$ and dimension less than $d$, therefore by induction $\pi_1^{\rm reg}(W',R_{W'}, {\bf N})$ has a solvable subgroup of length $2d-3$ and index $i(d-1,c)$.

    Therefore by Lemma~\ref{lem-virt-sol}, the group $ \pi_1^{\rm reg}(Z',R_{Z'},{\bf M})$ is solvable of length at most $2d-3+2$ and index at most $c(d)!i(d-1,c)$.

    As  $ \pi_1^{\rm reg}(Z',R_{Z'},{\bf M})$ surjects onto $ \pi_1^{\rm reg}(X,D)$, the desired bounds hold.
\end{proof}

\begin{proof}[Proof of Theorem~\ref{introthm-solv012}]
    By \cite[Theorem 3]{GLM23} there is an effective bound on the index of an abelian subgroup of $\pi_1^{\rm reg}(X,D)$ of klt Calabi-Yau surfaces. Therefore, the result follows in coregularity $\leq 2$ by Proposition~\ref{prop-induction-solvable}.
\end{proof}

\section{Fundamental Groups of klt Calabi Yau threefolds}

We start with the proof of the existence of `Torus Covers':

\begin{proof}[Proof of Theorem.~\ref{introthm-torus-covers}]
This is verbatim the proof of~\cite[Corollary~3.6]{GKP16b}, where we replace~\cite[Proposition~3.3]{GKP16b} with~\cite[Theorem~4.8]{Amb05} (see also~\cite{JX20}).    
\end{proof}

Now, we first prove finiteness of the fundamental group of the smooth locus in the terminal case.

\begin{theorem} 
\label{thm-terminal}
    Let $X$ be a normal projective threefold with terminal singularities. Let $K_X = 0$. Then either 
    \begin{enumerate}
        \item $X$ is smooth (in which case $\pi_1^{\rm reg}(X)$ is virtually abelian of rank $2\, \tilde{q}(X)$), 
        \item or $X$ is a finite quotient of an Abelian variety (in which case $\pi_1^{\rm reg}(X)$ is virtually abelian of rank six),
        \item or the fundamental group  $\pi_1^{\rm reg}(X)$ is finite.
    \end{enumerate}
    \end{theorem}

    \begin{proof}
        If $X$ is smooth, the statement is classical. So we assume $X$ is not smooth. Then $\tilde{q}(X)=0$, otherwise there would exist a finite quasi-\'etale cover $X' \to X$, again with terminal singularities, which splits as $X'=A \times Z$, with $A$ an Abelian variety and $Z$ having terminal singularities again. If $\tilde{q}(X)={\rm dim}(A) = 3$, we are in Case (2).
        
        So assume  $\tilde{q}(X)={\rm dim}(A) \neq 3$, but also ${\rm dim}(Z) \neq 1,2$, since terminal curve and surface singularities are smooth.

        Now, we consider a resolution of singularities $f : Y \to X$ and due to the finiteness of local fundamental groups, we know that there is an orbifold structure $\mathcal{Y}=(Y,\Delta=\sum (1-1/m_i) E_i)$ with $\Delta$ supported on the exceptional divisor of $f$, such that $\pi_1^{\rm orb}(Y,\Delta)=\pi_1^{\rm reg}(X)$, compare~\cite[Proof of Theorem~7]{Bra21}.
        By $\tilde{\mathcal{Y}} \to \mathcal{Y}$, we denote the associated orbifold universal cover, and we know that for the orbifold $L^2$-index and the orbifold Euler characteristic, we have the usual equality
        $$
        \chi_{(2)}(\tilde{\mathcal{Y}}, \mathcal{O}_{\tilde{\mathcal{Y}}})=\chi (\mathcal{Y},\mathcal{O}_{\mathcal{Y}}).
        $$
        The orbifold Euler characteristic of such a threefold has been computed in~\cite[Theorem~4.1]{TX17} in terms of the terminal isolated singularities of $X$ and Reid's invariants of these singularities, see~\cite{Rei87}. The formula looks as follows:
        $$
        \chi (\mathcal{Y},\mathcal{O}_{\mathcal{Y}})= \frac{-c_2(X) \cdot K_X}{24} + \sum_i \frac{1}{24 |I_i|} \sum_{j_i} \left( r_{ij_i}- \frac{1}{r_ij_i}\right) \geq \frac{-c_2(X) \cdot K_X}{24},
        $$
        where the index $i$ runs over all singularities $x_i$ of $X$, $I_i$ is their respective \emph{inertia group}, $j_i$ runs over all (cyclic quotient) singularities in the \emph{basket} of $x_i$, basically coming from a deformation of $x_i$, and $r_{ij}$ is the \emph{index of the singularity in the basket}.
        While $-c_2(X) \cdot K_X$ was positive in~\cite{TX17}, here clearly $-c_2(X) \cdot K_X = 0$, since $K_X$ is trivial. Now we have two cases: 
        \ \\
        \noindent
        \textbf{Case 1: No contributions from the singularities to $\chi (\mathcal{Y},\mathcal{O}_{\mathcal{Y}})$.}
        \ \\
        In this case, looking at how these contributions arise in~\cite{Rei87} or~\cite[Section~4]{TX17}, we see that
        $
        \pi_1^{\rm reg}(X)=\pi_1(X),
        $
        and for the latter, we know finiteness by~\cite[(4.17.3)]{Kol95}. So we only have to consider:
         \ \\
        \noindent
        \textbf{Case 2: There is a nontrivial contribution from some singularity to $\chi (\mathcal{Y},\mathcal{O}_{\mathcal{Y}})$.}
        \ \\
        Then $\chi (\mathcal{Y},\mathcal{O}_{\mathcal{Y}}) >0$. So $\chi_{(2)}(\tilde{\mathcal{Y}}, \mathcal{O}_{\tilde{\mathcal{Y}}}) >0$, and then we can finish as in~\cite[Proof of (4.1)]{Cam95}, since 
        $$
        \kappa^+(\mathcal{Y})={\rm \max}\{\kappa(\det(\mathcal{F}));~\mathcal{F} \mathrm{~is~a~coherent~subsheaf~of~}\Omega_{\mathcal{Y}}^p \mathrm{~for~some~}p\}=0
        $$ 
        by the same argumentation as in~\cite[Proof~of~Prop.~8.23]{GKP16b} or~\cite[Proof~of~Thm.~3.6]{Cam21}, which is valid since locally on the smooth locus, all objects defined for orbifolds agree with the usual ones.  
    \end{proof}

    Now, we can proceed to the general case:

    \begin{proof}[Proof of Theorem.~\ref{introthm-3dim}]
    Let $(X,D)$ be a threefold pair as in the theorem. By Theorem~\ref{introthm-torus-covers}, we can assume that $\tilde{q}(X)=0$.
        \ \\
        \noindent
        \textbf{Case 1: $D=0$ is trivial.}    
        \ \\
    Taking an index one cover, we can assume $K_X=0$ and that $X$ has canonical singularities. 
    Now we can take a terminalization as in~\cite[Lemma~4.9]{TX17}, the respective Euler characteristics and $L^2$-indices agree, and we can finish by Theorem~\ref{thm-terminal}.
     \ \\
        \noindent
        \textbf{Case 2: $D$ is nontrivial.}    
        \ \\
        Taking if necessary a small $\mathbb{Q}$-factorialization and by purity of the branch locus, we can assume that $X$ is $\mathbb{Q}$-factorial. Now we run a $K_X$-MMP which ends on a Mori fiber space $f: Z \to W$. By Lemma~\ref{lem-MMP-group}, we have a surjection 
        $$
        \pi_1^{\rm reg}(Z,D_Z)\twoheadrightarrow \pi_1^{\rm reg}(X,D),
        $$ 
        where $\varphi:X \dashrightarrow Z$ induces $\varphi_*(D)=D_Z$.

        If  $W$ is a point, then in case one of the coefficients of $D_Z$ is \textit{not} standard, then the standard approximation is a log Fano pair, and we are done by~\cite[Thm.~2]{Bra21}; in case all coefficients are standard, the orbifold index-one cover with respect to $K_Z+D_Z$ is a $K$-trivial canonical threefold, i.e. we can apply Case 1.  

         If on the other hand $\mathrm{dim}\,  W >0$, then we can apply Lemma~\ref{lem-fibration} to obtain the exact sequence

         $$
         \pi_1^{\rm reg}(F,D_F)\rightarrow \pi_1^{\rm reg}(Z,D_Z)\rightarrow \pi_1^{\rm reg}(W,D_W)\rightarrow 1,
         $$
         
         where the pair $(W,D_W)$ is the one obtained from the canonical bundle formula \cite[Theorem 4.1]{Amb05}. We have that $(F,D_F)$ is a log Fano pair, and $(W,D_W)$ is $klt$ Calabi-Yau of dimension at most $2$, thus a pair of absolute coregularity at most $2$. Hence, by Theorem~\ref{introthm-coreg12} with $c\leq 2$, $\pi_1^{\rm reg}(W,D_W)$ is virtually abelian of rank at most $4$. By Lemma~\ref{lem-virtual-groups}, the fundamental group $\pi_1^{\rm reg}(Z,D_Z)$ is virtually abelian of rank at most $4$, and so is $\pi_1^{\rm reg}(X,D)$.
        
    \end{proof}

\bibliographystyle{habbvr}
\bibliography{bib}

\end{document}